\documentclass{article}
\usepackage{amsmath,amssymb,amsthm}
\usepackage{hyperref}
\usepackage{cleveref}
\newtheorem{theorem}{Theorem}
\newtheorem{definition}[theorem]{Definition}
\newtheorem{conjecture}[theorem]{Conjecture}
\newtheorem{corollary}[theorem]{Corollary}
\newtheorem{lemma}[theorem]{Lemma}
\newtheorem{remark}[theorem]{Remark}
\usepackage{subcaption}
\usepackage[all]{xy}
\usepackage{graphicx}
\usepackage{booktabs}
\usepackage{url}

\usepackage{tikz}
\usetikzlibrary{matrix}
\usepackage{xcolor}
\usepackage[T1]{fontenc}
\usepackage[utf8]{inputenc}
\usepackage{mathrsfs}
\usepackage{flushend}
\usepackage{xspace}

\newcommand{\mt}{\mathtt}

\begin{document}
\title{Concatenative nonmonotonicity and optimal links in HP protein folding models}
\author{Bj{\o}rn Kjos-Hanssen}

\maketitle

\begin{abstract}
	The hydrophobic-polar (HP) model represents proteins as binary strings embedded in lattices, with fold quality measured by an energy score.
	We prove that the optimal fold energy is not monotonic under concatenation
	for several standard lattices, including the 2D and 3D rectangular, hexagonal, and triangular lattices.
	In other words, concatenating two polymers can produce a fold with strictly worse optimal energy than one of the polymers alone.
	
	For closed chains, we show that under the levels-of-hydrophobicity model of Agarwala et al. (1997), proper links can arise as uniquely optimal folds,
	revealing an unexpected connection between HP models and knot/link theory.
\end{abstract}

\section{Introduction}\label{sec:intro}
	Predicting the folded structure of a protein from its amino acid sequence has long been regarded as computationally intractable
	To make this precise, Ken A.~Dill~\cite{doi:10.1021/bi00327a032} introduced the hydrophobic-polar (HP) protein folding model in 1985.
	It was soon shown that finding an optimal fold in this model is NP-complete in both two and three dimensions~\cite{CGPPY:98,10.1145/279069.279080}.
	Despite remarkable empirical progress in protein structure prediction---most notably with Google's AlphaFold~\cite{alphafold}---the HP model
	remains a mathematically rich and fascinating abstraction.

	In the HP model, a protein is represented by a word $w$ over the alphabet $\{H,P\}$, where $H$ denotes hydrophobic and $P$ polar.  
	We will frequently identify $H=0$ and $P=1$.  
	A fold of $w$ is a self-avoiding walk in a lattice graph, with successive vertices labeled by the letters of $w$.  
	Each pair of non-consecutive $H$’s that become adjacent in the lattice contributes one point to the score:
	\[
	\xymatrix{
	H\ar@{-}[r] & P\ar@{-}[d]\\
	H\ar@{-}[r] & P\\		
	}
	\]
	An optimal fold is one that maximizes this score.

	We can view this as an example of parametric optimization. For a word $\theta\in\Theta= \{0,1\}^*$, and folds $x$, the \emph{optimal value} function $J$ is given by
	\[
		J(\theta) = \max_x f(x,\theta),
	\]
	where the \emph{fitness function} or \emph{objective function} $f$ gives the score
	of $\theta$ under the fold $x$. If we write $f = - E$ where $E$ stands for \emph{energy}, the task is to minimize energy.
	
	Equivalently, a fold $x$ may be described either as a sequence of occupied lattice sites, or as a sequence of moves between neighboring sites.

\paragraph{Automatic complexity.}
	There is an interesting similarity between the HP model and automatic complexity~\cite{KjosHanssen+2024}. 
	In the HP setting, each scored point arises from the polymer returning to a previously occupied lattice site. 
	This is analogous to an automaton revisiting a state. 
	In automatic complexity, minimizing the number of states encourages such revisits; in the HP model, maximizing the number of revisits corresponds (up to the hydrophobic/polar distinction) to minimizing the energy, which is simply the negative of the score.

\paragraph{Overview of the paper.}
	In \Cref{sec2}, we show that the optimal score in the HP model is non-monotone under concatenation,
	across several standard lattices (2D, 3D, triangular, and hexagonal).

	For the 2D rectangular and hexagonal lattices we can prove this follows from $Z$-optimality arguments,
	where $Z(w)$ is the number of hydrophobic monomers (zeros) in $w$.

	For 2D rectangular, 3D rectangular, and triangular lattice we use direct constructions and isoperimetric inequalities.

	These results are summarized in Table~\ref{refute}.

\begin{table}
	\begin{tabular}{l l l l}
		Lattice	&	Proof using $Z$-optimality	&Isoperimetric fact & Proof using that fact	\\
		\midrule
		triangular		&	not known			&\Cref{daisy}&	\Cref{iso_tri}\\
		hexagonal		&	\Cref{nov9-2024}&	not possible				\\
		2D rectangular	&	\Cref{key}			&\Cref{bl}	&	\Cref{iso-rect}\\
		3D rectangular	&	not known			&\Cref{bl3}	&	\Cref{nov23-2024_3D}\\
		\bottomrule
	\end{tabular}
	\caption{Summary of nonmonotonicity proofs for optimal folding in HP models across different lattices.}\label{refute}
\end{table}

In \Cref{sec:knots} we construct optimal closed folds in the 3D HP model that realize nontrivial knots and links. 
Finally, in \Cref{sec:indispensable} we show that under the \emph{levels of hydrophobicity} model of Agarwala et al.,
there are protein complexes for which proper links necessarily arise in every optimal folding.

\section{Nonmonotonicity}\label{sec2}

	\begin{definition}\label{df:J}
		Let $J_{\mathrm{rect}}(x), J_{\mathrm{cube}}(x), J_{\mathrm{tri}}(x), J_{\mathrm{hex}}(x)$ denote the maximum number of points achievable for a word $x$ in the
		2D and 3D rectangular, triangular, and hexagonal lattices, respectively.
	\end{definition}

	In a March 5, 2023, email message \cite{stecher}, Jack Stecher conjectured that the HP folding problem is ``weakly monotone'', i.e.,
	that whenever we add a prefix or suffix to a given word, our optimal score should never decrease.
	\begin{conjecture}[Stecher's \emph{cul-de-sac} conjecture]\label{conj}
		Let $x,y$ be binary words. Then $J_{\mathrm{rect}}(x)\le J_{\mathrm{rect}}(xy)$.
	\end{conjecture}
	Clearly the reverse $x^R$ of a word $x$ satisfies $J(x)=J(x^R)$. Therefore, we could equivalently conjecture that $J_{\mathrm{rect}}(y)\le J_{\mathrm{rect}}(xy)$.
	Thus, the conjecture would imply that for all $y,y'$,
	\begin{equation}\label{referee-last}
		J_{\mathrm{rect}}(x)
		\le J_{\mathrm{rect}}(xy')
		\le J_{\mathrm{rect}}(yxy').
	\end{equation}
	Biologically we may expect \Cref{conj} to fail: proteins can have a hydrophobic core in their interior, which is size-limited.
	To refute the conjecture we make use of \Cref{key}, which is also of independent interest.
	While \Cref{key} appears to be new, a special case with $Z=6$ was found by Matthew Gilzinger \cite[Figure 13]{Gilzinger2012}.

	\begin{theorem}[{Agarwala et al.~\cite[Lemma 2.1]{ABDDFHMS97}}]\label{agarwala}
		For all $w$, $J_{\mathrm{rect}}(1w1)\le Z(w)$, where $Z(w)$ is the number of zeros in $w$.
	\end{theorem}
	\begin{proof}
		Each zero in $x=1w1$ is internal, i.e., it is neither the first nor the last bit of $x$.
		Therefore, among the \textbf{four} ``heavenly'' directions from the location of the zero,
		two are occupied by the previous and the next amino acids, and the other \textbf{two} could each contribute to a point.
		As each edge is shared by two vertices, by the Handshaking Lemma each zero contributes on average at most \textbf{one} point.
	\end{proof}
	When the word may start or end with a zero the bound becomes:
	\begin{theorem}[{\cite[Fact 1]{AICHHOLZER2003139}}]
		For all $w$, $J_{\mathrm{rect}}(w)\le Z(w)+1$.
	\end{theorem}
	\begin{proof}
		There are now up to \textbf{three} available heavenly directions where a point may be earned at the first and last bits of $w$.
		The upper bound is therefore, with $w=w_1\dots w_n$, each $w_i\in\{0,1\}$,
		\[
			\frac{3\cdot 1_{w_1=0}+(\sum_{i=2}^{n-1}2\cdot 1_{w_i=0})+3\cdot 1_{w_n=0}}2
			\le \left(\sum_{i=1}^n 1_{w_i=0}\right) + 1
			= Z(w)+1.
		\]
		Here, the indicator symbol $1_P$ is 1 if $P$ holds and 0 otherwise.
	\end{proof}

\begin{theorem}\label{key}
	There are infinitely many words $w$ with $J_{\mathrm{rect}}(w)=Z(w)+1$.
\end{theorem}
\begin{proof}
	There are examples of each length of the form $26+8k$, $k\ge 0$:
	\begin{equation}\label{nov2024}
		(011)^3 1^{10} (0011)^k 0110 (1100)^k 110
	\end{equation}
	A suitable fold is defined by the $25+8k$ moves
	\[
	urdrdldrr u^4 l^4 dd (luld)^k ldr (drur)^k dru,
	\]
	where $u$ is ``up'', $l$ is ``left'', $d$ is ``down'', and $r$ is ``right''.
	The case $k=3$ is shown below.
\[
\xymatrix@C=1.7em@R=1.7em{
			&				&				&			  &				&			&			&1\ar@{-}[d]&	1\ar@{-}[l]	&	1\ar@{-}[l]	&	1\ar@{-}[l]	&	1\ar@{-}[l]\\
			& 1\ar@{-}[d]	& 1\ar@{-}[l]	& 1\ar@{-}[d] & 1\ar@{-}[l]	&1\ar@{-}[d]&1\ar@{-}[l]&1\ar@{-}[d]&	1\ar@{-}[r]	&	1\ar@{-}[d]	&	&	1\ar@{-}[u]\\
1\ar@{-}[d]	& 0\ar@{-}[l]	& 0\ar@{-}[u]	& 0\ar@{-}[l] & 0\ar@{-}[u]	&0\ar@{-}[l]&0\ar@{-}[u]&0\ar@{-}[l]&	*+[Fo]{0}\ar@{-}[u]	&	0\ar@{-}[r]	&	1\ar@{-}[d]	&	1\ar@{-}[u]\\
1\ar@{-}[r]	& 0\ar@{-}[d]	& 0\ar@{-}[r]	& 0\ar@{-}[d] & 0\ar@{-}[r]	&0\ar@{-}[d]&0\ar@{-}[r]&0\ar@{-}[d]&	*+[Fo]{0}		&	0\ar@{-}[d]	&	1\ar@{-}[l]	&	1\ar@{-}[u]\\
			& 1\ar@{-}[r]	& 1\ar@{-}[u]	& 1\ar@{-}[r] & 1\ar@{-}[u] &1\ar@{-}[r]&1\ar@{-}[u]&1\ar@{-}[r]&	1\ar@{-}[u]	&	1\ar@{-}[r]	&	1\ar@{-}[r]	&	1\ar@{-}[u]\\
}
\]
\end{proof}

\begin{corollary}\label{culdesacfalse}
	The \emph{cul-de-sac} conjecture is false.
\end{corollary}
\begin{proof}
	Suppose otherwise.
	Let $x$ be a word of the form~\eqref{nov2024}.
	Then we have the arithmetic contradiction
	\begin{align*}
		Z(x)+1 && = &&J_{\mathrm{rect}}(x) && \text{by \Cref{key},}\\
		       &&\le&& J_{\mathrm{rect}}(1x1) &&\text{by the \emph{cul-de-sac} conjecture \eqref{referee-last},}  \\
		    &&\le&& Z(x) &&\text{by \Cref{agarwala}.}
	\end{align*}
\end{proof}

The vertices of the \emph{hexagonal lattice} are vertices of hexagons and form a 3-regular graph.

\begin{theorem}\label{hex}
	For all $w$, $J_{\mathrm{hex}}(1w1)\le Z(w)/2$.
\end{theorem}
\begin{proof}
	Each zero in $x=1w1$ is internal, i.e., it is neither the first nor the last bit of $x$.
	Therefore, among the \textbf{three} ``heavenly'' directions from the location of the zero,
	two are occupied by the previous and the next amino acids, and the other \textbf{one} could contribute to a point.
	As each edge is shared by two vertices, by the Handshaking Lemma each zero contributes on average at most \textbf{one half} of a point.
\end{proof}
When the word may start or end with a zero the bound becomes:
\begin{theorem}
	For all $w$, $J_{\mathrm{hex}}(w)\le Z(w)/2+1$.
\end{theorem}
\begin{proof}
	There are now up to \textbf{two} available heavenly directions where a point may be earned at the first and last bits of $w$.
	The upper bound is therefore, with $w=w_1\dots w_n$, each $w_i\in\{0,1\}$,
	\[
		\frac{2\cdot 1_{w_1=0}+(\sum_{i=2}^{n-1}1\cdot 1_{w_i=0})+2\cdot 1_{w_n=0}}2
		\le \frac12\left(\sum_{i=1}^n 1_{w_i=0}\right) + 1
		= Z(w)/2+1.
	\]
\end{proof}
A hexagonal lattice fold achieving the maximum score $Z(w)/2+1=6$ in terms of the number of zeros $Z(w)=10$ in the folded word is shown below.
\begin{equation}\label{platypus}
\xymatrix@R=1em@C=1em{
	&	1\ar@{-}[dr]\ar@{-}[dl]	&	\\
1\ar@{-}[d]	&		&	1\ar@{-}[d]\\
1	&		&	*+[Fo]{0}	&		&	1\ar@{-}[dl]\ar@{-}[dr]	&		&	1\ar@{-}[dl]\ar@{-}[dr]	&		&	1\ar@{-}[dl]\ar@{-}[dr]	&		&	1\ar@{-}[dl]\ar@{-}[dr]	&	\\
	&	0\ar@{-}[ul]	&		&	*+[Fo]{0}	&		&	0	&		&	0	&		&	0	&		&	1\ar@{-}[d]\\
	&	1\ar@{-}[u]	&		&	0	&		&	0	&		&	0	&		&	0	&		&	1\ar@{-}[dl]\\
	&		&	1\ar@{-}[ul]\ar@{-}[ur]	&		&	1\ar@{-}[ul]\ar@{-}[ur]	&		&	1\ar@{-}[ul]\ar@{-}[ur]	&		&	1\ar@{-}[ul]\ar@{-}[ur]	&		&	1\ar@{-}[ul]	&\\
}
\end{equation}

\begin{theorem}\label{nov9-2024}
	For each integer $k\ge 1$, the word $w_k = 01^4 011 (01)^k 11 (10)^k$ satisfies $J_{\mathrm{hex}}(w_k)=Z(w_k)/2+1$.
\end{theorem}
\begin{proof}
	Using the fold in~\eqref{platypus} we obtain $J_{\mathrm{hex}}(w)=Z(w)/2+1$, which is the best possible by \Cref{hex}.
\end{proof}

\begin{corollary}\label{dec31cor}
	There are infinitely many counterexamples to the \emph{cul-de-sac} conjecture for the hexagonal lattice.
\end{corollary}
The proof of \Cref{dec31cor} has the same steps as the proof of \Cref{culdesacfalse}.

\begin{theorem}\label{nov10-24}
Let $x$ be a binary word and $J\in\{J_{\mathrm{rect}}, J_{\mathrm{tri}}, J_{\mathrm{hex}}\}$. Then $J(x1)\le J(x)$ and $J(1x)\le J(x)$.
\end{theorem}
\begin{proof}
	If $F$ is a fold achieving maximum score for $x1$ then we obtain a fold of $x$ with the same score simply by removing the final (respectively, initial) 1 from the fold.
\end{proof}

\Cref{nov10-24-II} provides one way of going from one counterexample to infinitely many, in the case of the \emph{cul-de-sac} conjecture.
In general, such an improvement is of interest because it necessarily requires insight beyond brute force computation.
\begin{theorem}\label{nov10-24-II}
	For $J\in\{J_{\mathrm{rect}}, J_{\mathrm{tri}}, J_{\mathrm{hex}}\}$, if $J(x1)<J(x)$ is witnessed by a fold $F$ of $x$ which can be extended to a fold of $1^m x$ for some $m\in\mathbb N$
	then $J(1^m x1) < J(1^m x)$ as well.
\end{theorem}
\begin{proof}
	Using \Cref{nov10-24} and our assumptions we have
	\[
		J(1^m x1) \le J(x1) < J(x) = J(1^m x).\qedhere
	\]
\end{proof}

\Cref{daisy} records the fact that if $n = 1 + 3r (r + 1)$ then $D_n=B_r$ (the closed ball with radius $r$, also known as the hexagonal daisy $D_n$) is a solution to the \emph{edge-isoperimetric problem} for cardinality $n$.
This means that the minimum number of edges connect vertices in the set $B_r$ to vertices outside the set, among all sets of cardinality $n$.

\begin{theorem}[{\cite[Corollary 7.1]{MR2035509}; see also~\cite{MR3614767}}]\label{daisy}
	For $n = 1 + 3r + 3r^2$, the hexagonal daisy $D_n$ is the unique minimizer of the edge-isoperimetric problem for the triangular lattice.
\end{theorem}

Given words $\eta, \theta$, we define $I_{\eta}(\theta)=\mathrm{occ}(\theta, \eta)$ to be the number of occurrences of $\eta$ in $\theta$.
For example, $I_{00}(001000)=3$.
The statistic $I_{00}(\theta)$ represents the number of internal ``hydrophobic'' connections of $\theta$.
\begin{theorem}\label{iso_tri}
There are infinitely many counterexamples to the \emph{cul-de-sac} conjecture for the triangular lattice.
\end{theorem}
\begin{proof}
	We shall exhibit for each $n\in\mathbb N$ a pair $(x,w)$ of a fold $x$ and a word $w$ such that
	\begin{itemize}
		\item $w$ contains no subword $0^{2n+1}$, and
		\item for each word $v$ with $Z(v)=Z(w)$,
	if $I_{00}(1v1)+f(y,1v1) \ge I_{00}(w)+f(x,w)$
	then $v$ contains a subword $0^{2n+1}$.
	\end{itemize}
	(We apply this with $v=w$.)
	The fold $x$ will be such that the zeros of $w$ are organized in a metric ball, which forces $y$ to do the same for $1v1$.
	
	Indeed, by \Cref{daisy}, for any words $\eta,\theta$ with $Z(\eta)=Z(\theta)$ and where $I_{00}(\zeta)\ge I_{00}(\eta)$,
	and any folds $x,y$, if $x$ has the zeros of $\eta$ in a metric ball then $f(x,\eta)\ge f(y,\zeta)$.

	This is because in general, if the zeros of a word $\zeta$ are organized in a metric ball by a fold $x$ then
	\[
	I_{00}(\zeta)+f(x,\zeta) = \max_{\eta : Z(\eta)=Z(\zeta)} I_{00}(\eta)+f(y,\eta).
	\]
\begin{figure}[t!]
	\centering
	    \begin{subfigure}[t]{0.5\textwidth}
	        \centering
\[
			\xymatrix@R=0.3em@C=0.3em{
	&&&1\ar@{-}[dr]&&&\\
			1\ar@{..>}[dr]\ar@{-}[rr]&&0\ar@{-}[ur]&&0\ar@{-}[rr]&&1\ar@{-}[dl]\\
			&0_b\ar@{-}[rr]&&0_a&&0\ar@{-}[dr]&\\
			1\ar@{-}[ur]&&0\ar@{-}[ll]&&0\ar@{-}[dl]&&1\ar@{-}[ll]\\
			&&&1\ar@{-}[ul]&&&\\
	}
\]	        \caption{Pair $(x,\zeta)$ with $I_{00}(\zeta)=1$,$f(x,\zeta))=11$, $\zeta=0(01)^6$.}
	    \end{subfigure}%
	    ~ 
	    \begin{subfigure}[t]{0.5\textwidth}
	        \centering
\[
			\xymatrix@R=0.3em@C=0.3em{
			\\
			&&0\ar@{-}[rr]&&0&&\\
			&0_b\ar@{-}[rr]\ar@{-}[dr]&&0_a&&0\ar@{-}[ul]&\\
			&&0&&0\ar@{-}[ll]\ar@{-}[ur]&&\\
			\\
}
\]
	        \caption{Pair $(y,\eta)$ with $I_{00}(\eta)=6$, $f(y,\eta)=6$, $\eta=0^7$.}
	    \end{subfigure}
\caption{Illustration of the Induced Edge Problem solution for the triangular lattice.}\label{whiskers}
\end{figure}
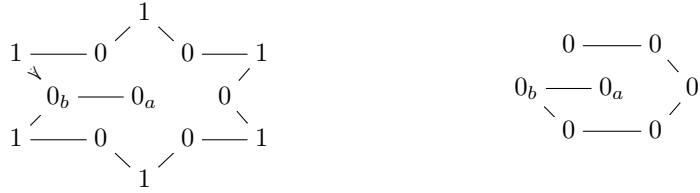

	For example, the pairs $(x,\zeta), (y,\eta)$ in \Cref{whiskers} both realize the maximum above.\footnote{This sum of $I_{00}$ and $f$
	is discussed in the second paragraph of Section 2 by Berger and Leighton, and
	on page 3 by Harper where he phrases it as: \emph{the Edge-Isoperimetric Problem and the Induced Edge Problem are equivalent and we shall treat them as interchangeable}.}
	
	Therefore, in any fold $y$ of $1v1$ with $f(y,1v1)=J(1v1)$, letting $y_0$ be the induced fold of $v$,
	we have $f(y_0,v) \ge f(y,1v1)\ge f(x,w)$ and so the zeros of $v$ must also be organized in a metric ball in $y_0$.
	Moreover, in order for $1v1$ to achieve enough points it must be that $Z(v)=Z(w)$.
	Write $v=u0w$ where the 0 between $u$ and $w$ is located in the center of the metric ball according to the fold $y$.
	Then the path from the initial 1 of $1v1$ to the final 1 of $1v1$ via the center 0 of the ball must contain a factor that enters the metric ball,
	proceeds to its center, and exits the metric ball.
	Thus, we must have $u=u' 0^n$ and $w = 0^n w'$ for some words $u',w'$.
	We conclude that $v$ contains a factor of the form $0^{2n+1}$.
	
	We create some folds of $0^*$ in the shape of a metric ball, and add \emph{whiskers}, replacing some edges $0\to 0$ on the boundary by $0\to 1^*\to 0$,
	in order to prevent long sequences of 0s. For example, the fold on the left in~\eqref{whiskers} has whiskers.

	We start from a symmetric closed-loop configuration and break its symmetry in order to connect the central position labeled $0_a$ and an adjacent position labeled $0_b$.
	See~\Cref{demonstrate}, \Cref{demonstrate5} for the pairs $P_n=(x_n,w_n)$, $n\le 5$.
\end{proof}
\begin{figure}
\[
	\begin{gathered}
		\xymatrix@R=0.3em@C=0.3em{
w_1= (001)^30&x_1:&&&&&&&&&1\ar@{-}[dr]&&&\\
&&&&&&&		&&0\ar@{-}[ur]&&0\ar@{-}[dr]&&\\
&&&&&&&		&0_b\ar@{-}[rr]&&0_a&&0\ar@{-}[dr]&\\
&&&&&&&		1\ar@{-}[ur]&&0\ar@{-}[ll]&&0\ar@{-}[ll]&&1\ar@{-}[ll]\\
w_2= (0001)^60 &x_2:&&&&						&					&				&				&				&				&	1\ar@{-}[dl]\ar@{-}[dr]			&	&	&	\\
&&&&&						&	1\ar@{-}[rr]	&				&	0\ar@{-}[dr]&							&		0\ar@{-}[dl]		&				&	0\ar@{-}[dl]			&		&	 \\
&&&&&						&					&	0\ar@{-}[ul]&				&	0			&				&	0\ar@{-}[rr]&			&	0	&	& 1\ar@{-}[dl]\ar@{-}[ll]\\
&&&&&						&	0\ar@{-}[rr]	&				&	0_b\ar@{-}[rr]&				&	0_a			&				&	0						&		&	0\ar@{-}[ll] \\ %
&&&&&1\ar@{-}[ur]\ar@{-}[rr]&					&	0\ar@{-}[rr]&				&	0						&				&	0			&							&	0\ar@{-}[ul]	&		&\\ %
&&&&&						&					&				&	0\ar@{-}[ur]&							&	0\ar@{-}[ur]&				&	0\ar@{-}[ul]\ar@{-}[rr]	&		&	1\ar@{-}[ul]\\ %
&&&&&						&					&				&				&	1\ar@{-}[ul]\ar@{-}[ur]	&				&				&							&		&	\\
w_3= 0^51(0^61)^500 &x_3:&&		&			&			&			&			&	&			&			&			&	1\ar@{-}[dl]&			&			&			&\\ 
&&&		&			&	1\ar@{-}[dr]&			&	0\ar@{-}[ll]&			&	0\ar@{-}[ll]&			&	0\ar@{-}[dl]&			&	0\ar@{-}[ul]&			&		&\\ 
&&&		&			&			&	0\ar@{-}[dr]&			&	0\ar@{-}[ur]&			&	0\ar@{-}[dl]&			&	0\ar@{-}[rr]&			&	0\ar@{-}[ul]&		&		&\\
&&&&			&	0\ar@{-}[dl]&			&0&			&	0\ar@{-}[ul]&			&	0\ar@{-}[ur]&			&	0\ar@{-}[ll]		&			&	0\ar@{-}[ll]	&	&	1\ar@{-}[ll]\\ 
&&&		&	0\ar@{-}[dl]&			&	0\ar@{-}[ul]&			&	0_b\ar@{-}[ll]&			&	0_a\ar@{-}[ll]&			&	0\ar@{-}[rr]&			&	0\ar@{-}[dr]&		&	0\ar@{-}[ur]	&			&\\ 
&&&1\ar@{-}[rr]&			&	0\ar@{-}[rr]&			&	0\ar@{-}[rr]&			&	0\ar@{-}[dl]&			&	0\ar@{-}[dr]&			&	0\ar@{-}[ul]&			&	0\ar@{-}[ur]	&		&	\\ 
&&&		&			&			&	0\ar@{-}[dr]&			&	0\ar@{-}[ll]&			&	0\ar@{-}[ur]&			&	0\ar@{-}[dl]&			&	0\ar@{-}[ul]&		&		&\\
&&&		&			&	&			&	0\ar@{-}[dr]&			&	0\ar@{-}[ur]&			&	0\ar@{-}[rr]&			&	0\ar@{-}[rr]&			&	1\ar@{-}[ul]	&\\ 
&&&		&			&			&			&			&	1\ar@{-}[ur]&			&			&			&	&			&			&		&\\
w_4 = 0^510^31(0^710^31)^50^3 &x_4:&			&			&			&			&			&			&			&	1\ar@{-}[dl]&			&			&			&	1\ar@{-}[dl]&	&	&	&	&	\\
&			&			&			&1\ar@{-}[dr]	&			&	0\ar@{-}[ll]&				&	0\ar@{-}[dl]&			&	0\ar@{-}[ul]&			&	0\ar@{-}[dl]&			&	0\ar@{-}[ul]&			&			&	&	\\ 
&			&			&			&			&0\ar@{-}[dr]	&			&0\ar@{-}[ul]		&			&	0\ar@{-}[ur]&			&	0\ar@{-}[dl]&			&	0\ar@{-}[ur]&			&	0\ar@{-}[ll]&		&	1\ar@{-}[ll]&	\\ 
&			&1\ar@{-}[dr]	&			&0\ar@{-}[ll]	&			&	0\ar@{-}[dr]&				&	0\ar@{-}[ur]&			&	0\ar@{-}[dl]&			&	0\ar@{-}[rr]&			&	0\ar@{-}[rr]&			&	0\ar@{-}[ur]&	&	\\ 
&			&			&0\ar@{-}[dr]	&			&0\ar@{-}[ul]	&			&0	&			&	0\ar@{-}[ul]&			&	0\ar@{-}[ur]&			&	0\ar@{-}[ll]&			&	0\ar@{-}[ll]&			&	0\ar@{-}[ll]&	& 1\ar@{-}[ll]\\ 
&			&0\ar@{-}[dl]	&			&0\ar@{-}[ll]	&			&	0\ar@{-}[ul]&				&0_b\ar@{-}[ll]	&			&0_a\ar@{-}[ll]	&			&	0\ar@{-}[rr]&			&	0\ar@{-}[dr]&			&	0\ar@{-}[rr]&	&	0\ar@{-}[ur]\\ 
&1\ar@{-}[rr]	&			&0\ar@{-}[rr]	&			&0\ar@{-}[rr]	&			&0\ar@{-}[rr]		&			&	0\ar@{-}[dl]&			&	0\ar@{-}[dr]&			&	0\ar@{-}[ul]&			&	0\ar@{-}[dr]&			&	0\ar@{-}[ul]&	\\ 
&			&			&			&0\ar@{-}[dl]	&			&	0\ar@{-}[ll]&				&	0\ar@{-}[ll]&			&	0\ar@{-}[ur]&			&	0\ar@{-}[dl]&			&	0\ar@{-}[ul]&			&0\ar@{-}[rr]	&	&	1\ar@{-}[ul]\\ 
&			&			&1\ar@{-}[rr]	&			&0\ar@{-}[rr]	&			&0\ar@{-}[dl]		&			&	0\ar@{-}[ur]&			&	0\ar@{-}[dl]&			&	0\ar@{-}[dr]&			&	0\ar@{-}[ul]&			&	&	\\ 
&			&			&			&			&			&	0\ar@{-}[dr]&				&	0\ar@{-}[ur]&			&	0\ar@{-}[dr]&			&	0\ar@{-}[ur]&			&	0\ar@{-}[rr]&			&1\ar@{-}[ul]	&	&	\\ 
&			&			&			&			&			&			&1\ar@{-}[ur]		&			&			&			&	1\ar@{-}[ur]&			&\\
}
	\end{gathered}
\]
\caption{Folds that demonstrate nonmonotonicity in the triangular model.}\label{demonstrate}
\end{figure}
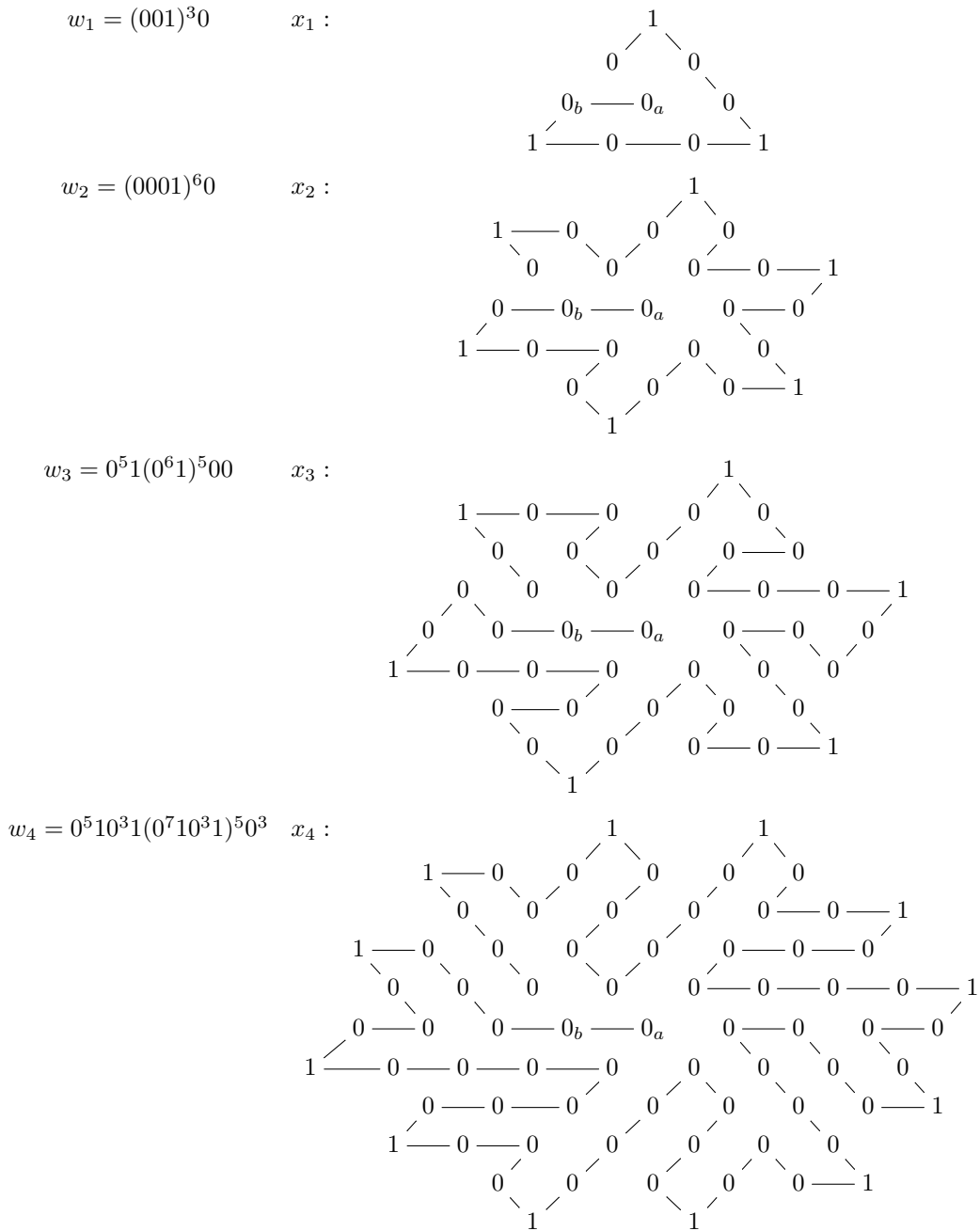

\begin{figure}
\[
\xymatrix@R=0.3em@C=0.3em{
		&			&			&			&			&			&			&	1\ar@{-}[dr]&			&			&			&			&			&	1\ar@{-}[dr]&			&			&			&			&			&			&			&			&		\\
		&			&			&			&			&			&	0\ar@{-}[ur]&			&	0\ar@{-}[rr]&			&	0\ar@{-}[dl]&			&	0\ar@{-}[ur]&			&	0\ar@{-}[dl]&			&	0\ar@{-}[rr]&			&	1\ar@{-}[dl]&			&			&			&		\\
		&			&			&	1\ar@{-}[rr]&			&	0\ar@{-}[dr]&			&	0\ar@{-}[ul]&			&	0\ar@{-}[dr]&			&	0\ar@{-}[ur]&			&	0\ar@{-}[dl]&			&	0\ar@{-}[ur]&			&	0\ar@{-}[dr]&			&			&			&			&		\\
		&			&			&			&	0\ar@{-}[ul]&			&	0\ar@{-}[dr]&			&	0\ar@{-}[ul]&			&	0\ar@{-}[ur]&			&	0\ar@{-}[dl]&			&	0\ar@{-}[ur]&			&	0\ar@{-}[dl]&			&	0\ar@{-}[ll]&			&			&			&		\\
		&			&			&	0\ar@{-}[dr]&			&	0\ar@{-}[ul]&			&	0\ar@{-}[dr]&			&	0\ar@{-}[ul]&			&	0\ar@{-}[dr]&			&	0\ar@{-}[ur]&			&	0\ar@{-}[rr]&			&	0\ar@{-}[rr]&			&	0\ar@{-}[rr]&			&	1\ar@{-}[dl]&		\\
1\ar@{-}[rr]&			&	0\ar@{-}[ur]&			&	0\ar@{-}[rr]&			&	0\ar@{-}[ul]&			&	0\ar@{-}[rr]&			&	0\ar@{-}[ul]&			&	0\ar@{-}[ur]&			&	0\ar@{-}[dl]&			&	0\ar@{-}[ll]&			&	0\ar@{-}[ll]&			&0\ar@{-}[ll]	&			&		\\
		&	0\ar@{-}[ul]&			&	0\ar@{-}[ll]&			&	0\ar@{-}[ll]&			&	0\ar@{-}[ll]&			&	0_b\ar@{-}[ll]&			&	0_a\ar@{-}[ll]&			&	0\ar@{-}[rr]&			&	0\ar@{-}[rr]&			&	0\ar@{-}[rr]&			&	0\ar@{-}[rr]&			&	0\ar@{-}[dr]&		\\
		&			&	0\ar@{-}[rr]&			&	0\ar@{-}[rr]&			&	0\ar@{-}[rr]&			&0&			&	0\ar@{-}[dl]&			&	0\ar@{-}[dr]&			&	0\ar@{-}[ll]&			&	0\ar@{-}[dr]&			&	0\ar@{-}[ll]&			&	0\ar@{-}[dl]&			&1\ar@{-}[ll]\\
		&	1\ar@{-}[ur]&			&	0\ar@{-}[ll]&			&	0\ar@{-}[ll]&			&	0\ar@{-}[ll]&			&	0\ar@{-}[dl]&			&	0\ar@{-}[ul]&			&	0\ar@{-}[dr]&			&	0\ar@{-}[ul]&			&	0\ar@{-}[dr]&			&	0\ar@{-}[ul]&			&			&		\\
		&			&			&			&	0\ar@{-}[rr]&			&	0\ar@{-}[ur]&			&	0\ar@{-}[dl]&			&	0\ar@{-}[ur]&			&	0\ar@{-}[dl]&			&	0\ar@{-}[dr]&			&	0\ar@{-}[ul]&			&	0\ar@{-}[dr]&			&			&			&		\\
		&			&			&			&			&	0\ar@{-}[ul]&			&	0\ar@{-}[dl]&			&	0\ar@{-}[ur]&			&	0\ar@{-}[dl]&			&	0\ar@{-}[ul]&			&	0\ar@{-}[dr]&			&	0\ar@{-}[ul]&			&	1\ar@{-}[ll]&			&			&		\\
		&			&			&			&	1\ar@{-}[ur]&			&	0\ar@{-}[ll]&			&	0\ar@{-}[ur]&			&	0\ar@{-}[dl]&			&	0\ar@{-}[ur]&			&	0\ar@{-}[ll]&			&	0\ar@{-}[dl]&			&			&			&			&			&		\\
		&			&			&			&			&			&			&			&			&	1\ar@{-}[ul]&			&			&			&			&			&	1\ar@{-}[ul]&			&			&			&			&			&			&		\\
}
\]
\caption{
The fold $x_5$ in the triangular lattice of a word $w_5$ having no occurrence of $0^{11}$.
}\label{demonstrate5}
\end{figure}
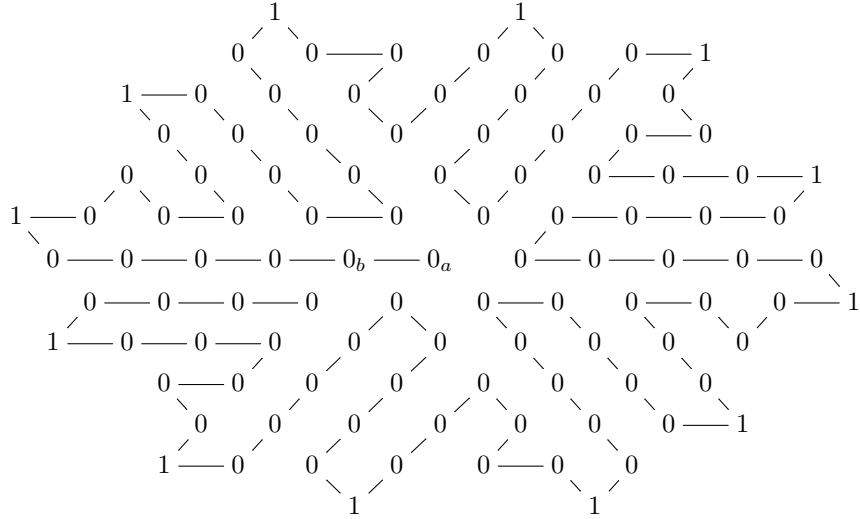

\subsection{2D rectangular lattice}

\begin{remark}
	While a metric ball $B_r$ optimizes the score for triangular folding, that is not true for rectangular folding.
	Compare the ball $B_3$ in~\Cref{nov13-2024} (which is a ball under the $L^1$ metric)
	to the square (or $L^{\infty}$ ball) of the same size $5^2 = 1+4+8+12$.
\end{remark}
\begin{figure}[t!]
	\centering
	    \begin{subfigure}[t]{0.5\textwidth}
	        \centering
\[
\xymatrix@R=1em@C=1em{
	&		&		&	3\ar@{-}[d]	&		&	\\
	&		&	3	&	2	&	3	&	\\
	&	3	&	2	&	1	&	2	&	3\\
3\ar@{-}[r]	&	2	&	1	&	0	&	1	&	2	&	3\ar@{-}[l]\\
	&	3	&	2	&	1	&	2	&	3\\
	&		&	3	&	2	&	3	&	\\
	&		&		&	3\ar@{-}[u]	&	&	\\
}
\]	        \caption{$B_3(L^1)$ has $25$ vertices and $36$ edges.}
	    \end{subfigure}%
	    ~ 
	    \begin{subfigure}[t]{0.5\textwidth}
	        \centering
\[\xymatrix@R=1em@C=1em{
\\
	&	4\ar@{-}[r]\ar@{-}[d]	&	3	&	2	&	3	&	4\ar@{-}[l]\ar@{-}[d]\\
	&	3	&	2	&	1	&	2	&	3\\
	&	2	&	1	&	0	&	1	&	2\\
	&	3	&	2	&	1	&	2	&	3\\
	&	4\ar@{-}[r]\ar@{-}[u]	&	3	&	2	&	3	&	4\ar@{-}[l]\ar@{-}[u]\\
	\\
}
\]
	        \caption{$B_2(L^{\infty})$ has 25 vertices and $40$ edges.}
	    \end{subfigure}
\caption{An optimal configuration (b) in the 2D rectangular lattice that is not a ball (a) in the graph metric.
The edges shown are the edges in the symmetric difference between the two graphs.}\label{nov13-2024}
\end{figure}
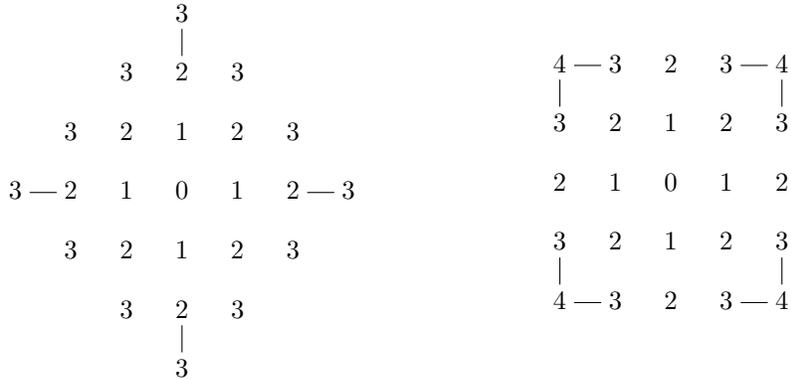

\begin{figure}[t!]
	\centering
	    \begin{subfigure}[t]{0.5\textwidth}
	        \centering
\[
\xymatrix@R=1em@C=1em{
P_2^{\text{rect}}\\
	&	0\ar[d]	&	0\ar[l]	&	0\ar[l]	&	0\ar[d]	&	0\ar[l]\\
	&	0\ar[r]	&	0		&	0\ar[u]	&	0\ar[l]	&	0\ar[u]\\
	&	0\ar[d]	&	0\ar[l]	&	0\ar[l]	&	0\ar[r]	&	0\ar[u]\\
	&	0\ar[d]	&	0\ar[r]	&	0\ar[d]	&	0\ar[u]	&	0\ar[l]\\
	&	0\ar[r]	&	0\ar[u]	&	0\ar[r]	&	0\ar[r]	&	0\ar[u]\\
	\\
}
\]	        \caption{Before whiskers.}
	    \end{subfigure}%
	    ~ 
	    \begin{subfigure}[t]{0.5\textwidth}
	        \centering
\[
\xymatrix@R=1em@C=1em{
	&			&	1\ar[d]		&	1\ar[l]		&	1\ar[d]		&	1\ar[l]		\\
1\ar[d]	&	0\ar[l]	&	0\ar[l]	&	0\ar[u]	&	0\ar[d]	&	0\ar[u]\\
1\ar[r]	&	0\ar[r]	&	0	&	0\ar[u]	&	0\ar[l]	&	0\ar[u]	&	1\ar[l]\\
1\ar[d]	&	0\ar[l]	&	0\ar[l]	&	0\ar[l]	&	0\ar[r]	&	0\ar[r]	&	1\ar[u]\\
1\ar[r]	&	0\ar[d]	&	0\ar[r]	&	0\ar[d]	&	0\ar[u]	&	0\ar[l]	&	1\ar[l]\\
	&	0\ar[d]	&	0\ar[u]	&	0\ar[d]	&	0\ar[r]	&	0\ar[r]	&	1\ar[u]\\
	&	1\ar[r]	&	1\ar[u]		&	1\ar[r]		&	1\ar[u]\\
}
\]
	        \caption{With whiskers.}
	    \end{subfigure}
\caption{Folding to obtain nonmonotonicity in the 2D rectangular model.}\label{nov27-2024}
\end{figure}

\begin{figure}
\[
\xymatrix@R=1em@C=1em{
P_3^{\text{rect}}		&			&			&	1\ar[d]	&	1\ar[l]	&	1\ar[d]	&	1\ar[l]	&			&	\\
		&	0\ar[d]	&	0\ar[l]	&	0\ar[d]	&	0\ar[u]	&	0\ar[d]	&	0\ar[u]	&	0\ar[l]	&	\\
1\ar[d]	&	0\ar[l]	&	0\ar[u]	&	0\ar[l]	&	0\ar[u]	&	0\ar[d]	&	0\ar[r]	&	0\ar[u]	&	\\
1\ar[r]	&	0\ar[r]	&	0\ar[r]	&	0\ar@{..>}[d]&0\ar[u]&	0\ar[l]	&	0\ar[u]	&	0\ar[l]	&1\ar[l]	\\
1\ar[d]	&	0\ar[l]	&	0\ar[l]	&	0\ar[l]	&	0\ar[l]	&	0\ar[r]	&	0\ar[r]	&	0\ar[r]	&1\ar[u]	\\
1\ar[r]	&	0\ar[r]	&	0\ar[d]	&	0\ar[r]	&	0\ar[d]	&	0\ar[u]	&	0\ar[l]	&	0\ar[l]	&1\ar[l]	\\
		&	0\ar[d]	&	0\ar[l]	&	0\ar[u]	&	0\ar[d]	&	0\ar[r]	&	0\ar[d]	&	0\ar[r]	&1\ar[u]	\\
		&	0\ar[r]	&	0\ar[d]	&	0\ar[u]	&	0\ar[d]	&	0\ar[u]	&	0\ar[r]	&	0\ar[u]	&	\\
		&			&	1\ar[r]	&	1\ar[u]	&	1\ar[r]	&	1\ar[u]	&			&			&	\\
}
\]
\caption{Folding to obtain nonmonotonicity in the 2D rectangular model.
}\label{follows}
\end{figure}

\begin{theorem}\label{nov23-2024}
	For each $n$ there exists a word $w_n$ and a fold $P^{\mathrm{rect}}_n$ of $w_n$ in the rectangular lattice such that the 0s of $w_n$ form a $(2n+1)$-square
	and $w_n$ contains no subword $0^{2n+1}$.
\end{theorem}
\begin{proof}[Proof sketch]
	We construct a suitable fold $P_n^{\text{rect}}$ of a word starting with $0^{n+1}1$ in a square.
	See \Cref{follows} for an illustration of $P^{\mathrm{rect}}_5$.
	For an illustration of $P_2^{\text{rect}}$, before and after adding whiskers, see \Cref{nov27-2024}.
\end{proof}
Now we use some classical results:
\begin{theorem}[{\cite[Fact 3.1]{10.1145/279069.279080}}]\label{bl3}
	Among all sets of cardinality $n^3$ in the rectangular lattice, an $n$-cube strictly maximizes the number of internal edges.
\end{theorem}
\begin{theorem}\label{bl}
	Among all sets of cardinality $n^2$ in the rectangular lattice, an $n$-square strictly maximizes the number of internal edges.
\end{theorem}
\begin{proof}
	The proof is the same as that of \Cref{bl3}.
\end{proof}
\begin{theorem}\label{iso-rect}
	There are infinitely many counterexamples to the \emph{cul-de-sac} conjecture in the rectangular lattice.
\end{theorem}
\begin{proof}
	Let $w_n$ be as in \Cref{nov23-2024}.
	Assume for contradiction that $J_{\mathrm{rect}}(1w_n1)=J_{\mathrm{rect}}(w_n)$. Thus, there is a fold of $1w_n1$ achieving as many points as $w_n$.
	This fold must include a $(2n+1)$-square of 0s by \Cref{bl}.
	Since $1w_n1$ starts and ends with 1s, the fold must start and end outside this $(2n+1)$-square.
	In order to reach the 0 at the center of the $(2n+1)$-square and return outside the square, the word $w_n$ must contain a factor $0^{2n+1}$.
	However, by construction $w_n$ does not contain such a factor.
\end{proof}

\subsection{3D rectangular lattice}
A similar construction succeeds for the 3D rectangular lattice.

\begin{figure}
\[
\xymatrix@R=1em@C=0.7em{
&&				&				&&			&1\ar@{-}[r]&1\ar@{-}[d]&				&				&&			&1\ar@{-}[r]&1\ar@{-}[d]&				&				&&			&1\ar@{-}[r]&1\ar@{-}[d]&				&				&\\
1\ar@{-}[r]&1\ar@{-}[d]&		&		&&			&0\ar@{-}[u]\ar@{-}@/_2pc/[llllll]&0\ar@{-}[r]&	0\ar@{-}[r]	&	1\ar@{-}[d]	&&			&0\ar@{-}[u]&0\ar@{-}[r]&	0\ar@{-}[r]	&	1\ar@{-}[d]	&&			&0\ar@{-}@/_2pc/[llllll]\ar@{-}[u]&0\ar@{-}[r]&	0\ar@{-}[r]	&	1\ar@{-}[d]	&\\
			&1\ar@{-}@/_2pc/[rrrrrr]			&		&		&&1\ar@{-}[r]&0			&0\ar@{-}@/_2pc/[rrrrrr]			&	0\ar@{-}[d]	&	1\ar@{-}[l]	&&1\ar@{-}[r]&0\ar@{-}@/_2pc/[llllll]			&0			&	0\ar@{-}[d]	&	1\ar@{-}[l]	&&1\ar@{-}[r]&0	&0\ar@{-}[l]	&	0\ar@{-}[d]	&	1\ar@{-}[l]	&\\
&&		&				&&1\ar@{-}[u]&0\ar@{-}[l]&0\ar@{-}[l]&	0\ar@{-}[d]	&				&&1\ar@{-}[u]&0\ar@{-}[l]&0\ar@{-}[l]&	0\ar@{-}[d]	&				&&1\ar@{-}[u]&0\ar@{-}[l]&0\ar@{-}[l]&	0\ar@{-}[d]	&				&\\
			&&	&				&&			&			&1\ar@{-}[u]&	1\ar@{-}[l]	&				&&			&			&1\ar@{-}[u]&	1\ar@{-}[l]	&				&&			&			&1\ar@{-}[u]&	1\ar@{-}[l]	&				&\\
}
\]
\includegraphics[width=12cm]{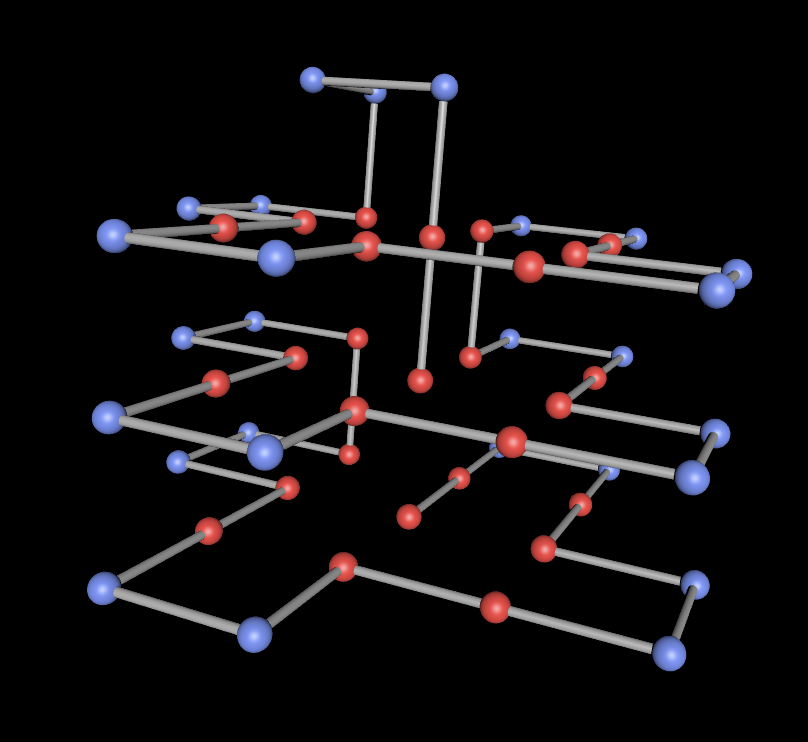}
\caption{Folding to obtain nonmonotonicity in the 3D rectangular model for the word $w=(0011)^{12}011100$.
For a draggable view of the graphics see \url{https://math.hawaii.edu/~bjoern/?sheets=6&xs=17&ys=6&string=001100110011001100110011001100110011001100110011011100&page=labbyfold&moves=aasddsdwwdwaawasewdsddsassawaawdeasddsdwwdwaawasedsff}.
}\label{follows3}
\end{figure}

\begin{theorem}\label{nov23-2024_3D}
For each $n\in\mathbb N$ there exists a word $w_n$ and a fold $P_n$ of $w_n$ in the 3D rectangular lattice such that the 0s of $w_n$ form an $n$-ball
and $w_n$ contains no subword $0^{2n+1}$. Thereby $w_n$ and $w_n 0$ form a counterexample to the \emph{cul-de-sac} conjecture.
\end{theorem}
\begin{proof}[Proof sketch]
	Using the $0^3$-free length-54 word
	\[
		(0011)^{12}011100 = 00 (1100)^{11} 11 01 (1100)^1
	\]
	a suitable fold is shown in \Cref{follows3}.
\end{proof}

\subsection{Berger-Leighton inspired method}
We can also refute the \emph{cul-de-sac} conjecture in an easier way.
Consider the following diagram.
\[
\xymatrix@R=0.7em@C=0.7em{
		&1\ar[d]	&	1\ar[l]	&1\ar[d]	&1\ar[l]	&	1\ar[d]		&1\ar[l]	&		\\
1\ar[d]	&0\ar[l]	&	0\ar[u]	&0\ar[l]	&0\ar[u]	&	0\ar[l]		&0\ar[u]	&1\ar[l]	\\
1\ar[r]	&0\ar[d]	&	0\ar[r]	&0\ar[r]	&0\ar[r]	&	0\ar[d]		&0\ar[r]	&1\ar[u]	\\
1\ar[d]	&0\ar[l]	&	0\ar[u]	&0\ar[r]	&0\ar[d]	&	0\ar[d]		&0\ar[u]	&1\ar[l]	\\
1\ar[r]	&0\ar[d]	&	0\ar[u]	&0\ar[u]	&0\ar[d]	&	0\ar[d]		&0\ar[r]	&1\ar[u]	\\
1\ar[d]	&0\ar[l]	&	0\ar[u]&0\ar[l]		&0\ar[l]	&	0\ar[d]		&0\ar[u]	&1\ar[l]	\\
1\ar[r]	&0\ar[d]	&	0\ar[r]	&0\ar[d]	&0			&	0\ar[d]		&0\ar[r]	&1\ar[u]	\\
		&1\ar[r]	&	1\ar[u]	&1\ar[r]	&1\ar[u]	&	1\ar[r]		&1\ar[u]	&		\\
}
\]
Here, the central square is labeled with 0s throughout and the boundary of the square is labeled with 1s.
In any optimal fold, the 1s would have to be on the boundary and completely cover the boundary.
Therefore, we have nonmonotonicity in the form $J(w)>J(w1)$ for the word $w=0^{16}0110((1100)^2 110)^31100110$.

\section{Knots and links}\label{sec:knots}
We do not know whether knots are necessary for the optimal value function $J_{\mathrm{cube}}$.
W\"ust, Reith and Virnau \cite{PhysRevLett.114.028102} considered knotting in the HP model, but only from a statistical point of view, hence they do not address
whether knotting is necessary for optimal folding.

Of course, an optimal closed fold can include any knot, simply by folding a word of the form $1^*$ in the shape of that knot.
Here a \emph{closed} fold is a walk in the lattice that is self-avoiding except that the first and last
vertices coincide.

Let us therefore say, given a knot $K$, that a \emph{proper $K$-fold} of a word $w$ is a closed fold of $w$ which looks like $K$ (i.e., is ambiently isotopic to $K$)
and such that in a 2-dimensional knot diagram of the fold, the over-and-under vertices are separated by points scored with respect to $w$.

Here is a diagram of the trefoil knot with over-and-under vertices indicated by curving.
\begin{center}
\(
\xymatrix{
\ar[ddd]			&	&				&	\ar[lll]			\\
			&	&	\ar@/^/[rr]	&	&	\ar[ddd]		\\
			&\ar@/^/[rr]	&				&\ar[uu]	&		\\
\ar@/^/[rr]	&	&	\ar[uu]		&	&		\\
			&\ar[uu]	&				&	&	\ar[lll]	\\
}
\)
\includegraphics[width=8cm]{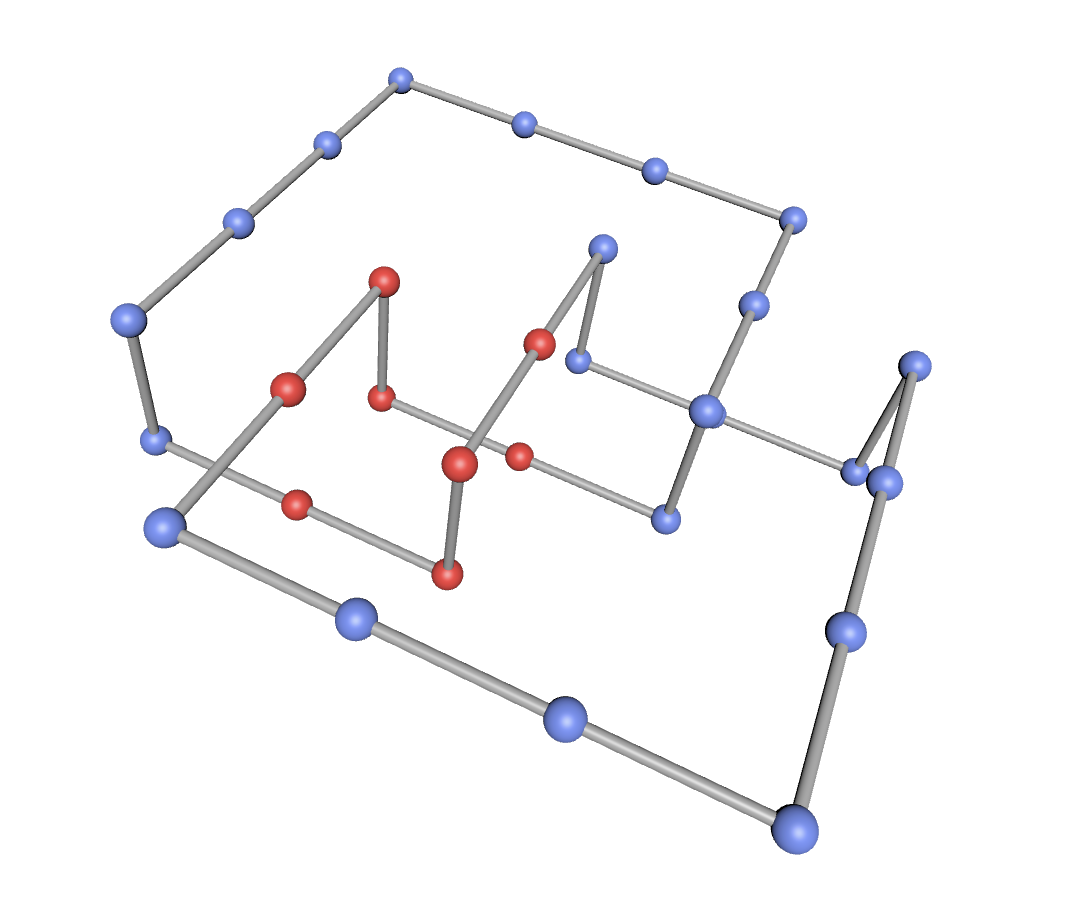}
\end{center}

It is immediate that only if $K$ is \emph{the unknot} can there be a proper $K$-fold of a word of the form $1^*$.

\begin{theorem}
	There exists a cyclic word $w$, a nontrivial knot $K$, and a proper $K$-fold $x$ of $w$ such
	that $x$ is an optimal fold of $w$.
\end{theorem}
\begin{proof}
	Let $w=1^2 (0^4 1^7)^2$, a word of length 24, let $K$ be the trefoil knot, and
	let $x$ be the fold shown below.
	As seen earlier, the central red cube (where red is identified with 0 and blue with 1) guarantees an optimal fold.
\end{proof}

\includegraphics[width=12cm]{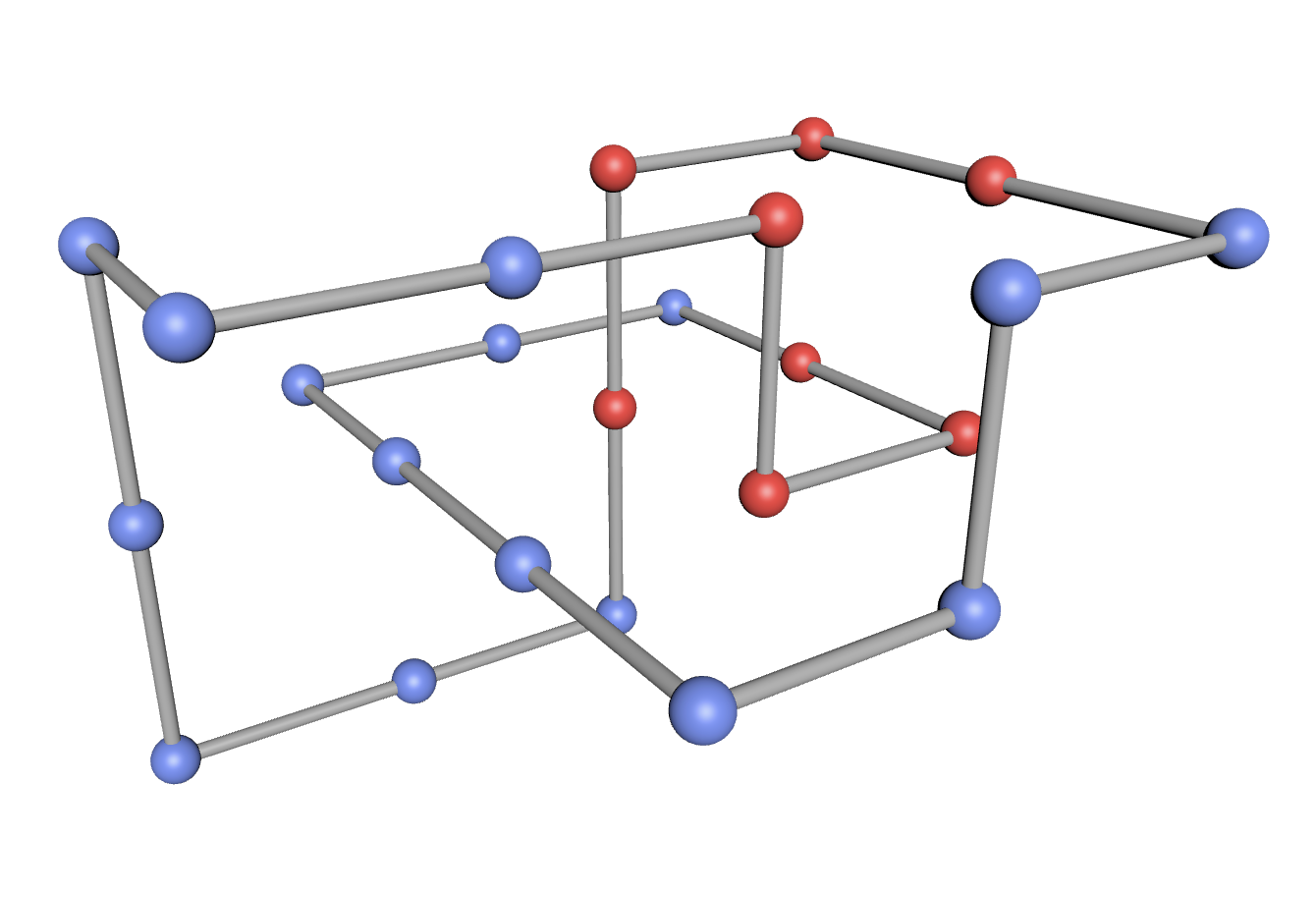}
\begin{verbatim}
https://math.hawaii.edu/~bjoern/?sheets=9&xs=17&ys=17
&string=1100001111111000011111111
&page=labbyfold&moves=eddfwwffssaeeweddffaaaes
\end{verbatim}

It is clear that we can achieve a central packed cube of 0s with a minimal length knot. The minimal length of a trefoil knot in the 3D integer lattice is 24 \cite{MR3265402}.
This was already achieved above for the word
$1^2 (0^41^7)^2$ and 6 points.
However, this is not uniquely optimal as we can fairly easily take the given configuration of red monomers and add 7 and 9 blue monomers (using 8 and 10 edges) in a non-knotted manner.
Thus, we know that a knot can appear in a minimal-length, optimal-score, but not-uniquely-optimal-score, manner in the 3D HP model.

\paragraph{An optimal link.}
Below we have a 4-cube formed from two linked loops of length 8 and $64-8=56$, respectively. We can make every monomer a 0 (hydrophobic)
but it is more interesting to only make a certain 3-cube that way. Then the words read are $0^8$
and (starting from the linking point) $0^6 1^{20} 0^2 1^2 0^2 1^9 0^2 1^4 0^2 1^2 0^5$.
It is not hard to see that these two words can also form a 4-cube in an unlinked way, however.
\[
\xymatrix@R=0.2em@C=2.0em{
		&			&			&	1\ar[dr]&			&				\\
		&			&	1\ar[dl]&			&	1\ar[dr]&				\\
		&	1\ar[dr]&			&	1\ar[ul]&			&	1\ar[dr]	\\
1\ar[dddddddd]&			&	1\ar[dr]&			&	1\ar[ul]&			&	1\ar[dl]\\
		&	1\ar[ul]&			&	1\ar[dr]&			&	1\ar[ul]&		\\
		&			&	1\ar[ul]&			&	1\ar[dl]&				\\
		&			&			&	1\ar[ul]&			&				\\
\\
		&			&			&	1\ar@/^2pc/[uuuuuuuu]&			&				\\
		&			&	1\ar[dr]&			&	0\ar[ul]&				\\
		&	1\ar[ur]&			&	0\ar[dl]&			&	0\ar[ul]	\\
1\ar[ur]&			&	0\ar[dl]&			&	0\ar[dr]&			&	0\ar[ul]\\
		&	1\ar[dr]&			&	0\ar[dr]&			&	0\ar[ur]&		\\
		&			&	1\ar[ur]&			&	0\ar[dl]&				\\
		&			&			&	1\ar@/^2pc/[dddddddd]&			&				\\
		&			&			&			&			&			&			\\
		&			&			&	1\ar@/_2pc/[dddddddd]&			&				\\
		&			&	1\ar[ur]&			&	0\ar@{-}[dr]&				\\
		&	1\ar[ur]&			&	0\ar@{-}[ur]&			&	0\ar@{-}[dr]	\\
1\ar[ur]&			&	0\ar@{-}[ur]&			&	0\ar@/_2pc/[uuuuuuuu]&			&	0\ar@{-}[dl]\\
		&	1\ar[ul]&			&	0\ar@{-}[ul]&			&	0\ar@{-}[dl]&		\\
		&			&	1\ar[ul]&			&	0\ar@{-}[ul]&				\\
		&			&			&	1\ar[ul]&			&				\\
		&			&			&			&			&			&			\\
		&			&			&	1\ar[dr]&			&				\\
		&			&	1\ar[dl]&			&	0\ar[dl]&				\\
		&	1\ar[dl]&			&	0\ar[ul]&			&	0\ar[dl]	\\
1\ar[dr]&			&	0\ar[dr]&			&	0\ar@/^2pc/[uuuuuuuu]&			&	0\ar[ul]\\
		&	1\ar[ur]&			&	0\ar[dl]&			&	0\ar[ur]&		\\
		&			&	1\ar[dr]&			&	0\ar[ur]&				\\
		&			&			&	1\ar[ur]&			&				\\
		&			&			&			&			&			&			\\
}
\]
\paragraph{An optimal trefoil knot.}
Since the 0s form a $4\times 4\times 4$ grid, the following is an optimal fold of the length-70 cyclic word
\(
1^3 0^{22} 1^3 0^{42}
\)
in the shape of a trefoil knot.
\[
\xymatrix@R=0.3em@C=2.0em{
		&			&			&	0\ar[dr]&			&			&		\\
		&			&	0\ar[ur]&			&	0\ar[dr]&			&		\\
		&	0\ar[dl]&			&	0\ar[ul]&			&	0\ar[dr]&		\\
0\ar[dr]&			&	0\ar[ur]&			&	0\ar@/^2pc/[dddddddd]		&			&	0\ar[dl]	\\
		&	0\ar[ur]&			&	0\ar[ur]&			&	0\ar[dl]&		\\
		&			&	0\ar[ur]&			&	0\ar[dl]&			&		\\
		&			&			&	0\ar[ul]&			&			&		\\
\\
		&						&			&	0\ar[dr]&			&			&		\\
		&						&	0\ar[ur]&			&	0\ar[dr]&			&		\\
		&	0\ar@/_2pc/[uuuuuuuu]	&			&	0\ar[ul]&			&	0\ar[dr]&		\\
0\ar[ur]&						&	0\ar[ur]&			&	0\ar[dl]&			&	0\ar@/^/[dddddddd]	\\
		&	0\ar[ur]			&			&	0\ar[dl]&			&	0\ar@/^/[dddddddd]		&		\\
		&						&	0\ar[dr]&			&	0\ar[ur]&			&		\\
		&						&			&	0\ar[ur]&			&			&		\\
\\
					&						&			&	0\ar[dl]&			&			&		\\
					&						&	0\ar[dr]&			&	0\ar[ul]&			&		\\
					&	0\ar[dl]			&			&	0\ar@/^2pc/[dddddddd]		&			&	0\ar[ul]&		\\
0\ar@/^2pc/[uuuuuuuu]	&						&	0\ar[ul]&			&	0\ar[ur]&			&	0\ar[dr]&			&		&\\
					&	0\ar@/^2pc/[uuuuuuuu]	&			&	0\ar[ul]&			&	0\ar[ul]&			&	1\ar[dl]&	\\
					&						&	0\ar[ur]&			&	0\ar[dl]&			&	1\ar[dl]&			&\\
					&						&			&	0\ar[ul]&			&	1\ar[ul]&			&			&		\\
					&						&			&			&			&			&			&			&\\
					&						&			&	0\ar[dr]&			&	1\ar[dr]&		\\
					&						&	0\ar[ur]&			&	0\ar[ur]&			&	1\ar[dr]&	\\
					&	0\ar[dl]			&			&	0\ar[ul]&			&	0\ar[dl]&			&1\ar[dl]	\\
0\ar[dr]			&						&	0\ar[ul]&			&	0\ar[dr]&			&	0\ar[ul]	\\
					&	0\ar@/_2pc/[uuuuuuuu]	&			&	0\ar[ul]&			&	0\ar[dl]&		\\
					&						&	0\ar[ur]&			&	0\ar[dl]&			&		\\
					&						&			&	0\ar[ul]&			&			&		\\
}
\]
Let us say that the \emph{hydrophobic core} of a fold is the map sending locations in the word to locations of the hydrophobic monomers, that
is undefined on polar monomers.
Let us moreover say that an \emph{essential fold} of the knot $K$ is a fold such that all folds with the same hydrophobic core are homotopic to $K$.
Then the above example shows that an essential fold of the trefoil knot $3_1$ can be an optimal fold of a word.
We conjecture that this is true for all knots.

\section{The indispensability of links}\label{sec:indispensable}
Agarwala et al.~\cite[Section 4]{ABDDFHMS97} added varying levels of hydrophobicity to the HP model.
In this expanded model we now demonstrate the existence of a link of four chains which achieves a higher
score than any non-linked ensemble.
Our example uses a multiset of words over the alphabet $\{\mt 0,\mt 1, \mt 2\}$, namely
\[
	\mathscr M_m = \{(\mt{01})^4, (\mt{01})^4, (\mt{01})^4, \mt{2}^3\mt{1}^{9+2m}\}
\]
for any $m\ge 0$.
We will focus on the case $m=0$, but it will be easy to see that any $m$ will give the same conclusion and hence we have infinitely many examples.

We will use $[x]_k\mathbb Z/k\mathbb Z$ to denote the congruence class of $x\in\mathbb Z$ mod $k$.
When $k$ can be inferred from context we write $[x]=[x]_k$.

\begin{definition}\label{def:emb}
A \emph{closed-chain embedding} $\psi$ of a multiset $\mathscr M$ of words over an alphabet $\Sigma$ consists of:
for each word $w$ occurring $k_w$ times in $\mathscr M$, a family of $k_w$ many injective functions
\[
	f_{w,i} : \mathbb Z/|w|\mathbb Z \to \mathbb Z^3, \quad 1 \le i \le k_w,
\]
each having the property that $d(f_{w_i}([x]),f_{w_i}([x+1]))=1$ for all $x$,
such that the ranges of the functions $f_{w,i}$ as $w$ and $i$ vary are all disjoint.

The \emph{contents} $\alpha(x, \psi)$ of a point $x\in\mathbb Z^3$ under $\psi$ is the symbol $a\in\Sigma$,
if any, such that for some $w$ and $i,k$, $f_{w,i}(k)=x$ and $w(k)=a$.

The set $\mathscr C(\psi)$ of \emph{potential contacts}
consists of sets $s=\{x,y\}$ where $x,y\in\mathbb Z^3$ are adjacent ($d_1(x,y)=1$) and have nonempty contents,
and such that for no $w$ and $i,k$ is $\{x,y\}=\{f_{w,i}([k]), f_{w,i}([k+1])\}$.
\end{definition}
While Dill's original model concerns open chains,
closed chains in HP models are topologically natural and have been studied by Aichholzer et al~\cite{AICHHOLZER2003139}.

\begin{definition}[Levels-of-hydrophobicity HP model~\cite{ABDDFHMS97}]
In the context of \Cref{def:emb},
each symbol $a\in\Sigma$ is assigned a \emph{level of hydrophobicity} $h(a)\in\mathbb N$.
The contribution (or value) of a pair of symbols $\{a,b\}$ is
\[
v_h(\{a,b\})= \begin{cases}
	h(a) + h(b) & \text{if }h(a)\ne 0\ne h(b),\\
	0	  & \text{otherwise}.
	\end{cases}
\]
If the value of a potential contact is positive then we call it a \emph{contact}.
The contribution of a pair of lattice points $x,y\in\mathbb Z^3$ is defined by
\[
	v'_h(\{x,y\})=v_h(\{\alpha(x,\psi),\alpha(y,\psi)\}).
\]
The \emph{score} of $\psi$ for $\mathscr M$ is
\[
	f(\psi,\mathscr M,h)=\sum_{s\in\mathscr C} v'_h(s),
\]
the sum of all contributions.
The optimal value function $J(\mathscr M,h)$ is the maximum of $f(\psi, \mathscr M,h)$ over all closed-chains embeddings $\psi$.
\end{definition}

In order to obtain the indispensability of links we will use the hydrophobicity level function $h_c$ with $c\ge 10$ (ten).
\begin{definition}
Let $h_c : \{\mt 0,\mt 1,\mt 2\} \to \mathbb N$ be given by
\[
	h_c(\mt 1)=0,\quad h_c(\mt 0)=1,\quad\text{and}\quad h_c(\mt 2)=c.
\]
\end{definition}

In the context of the hydrophobicity level function $h_c$ with $c>1$,
we will refer to contributions of value $c+1$ and 2 as \emph{$c$-contributions} and \emph{1-contributions}, respectively.
\Cref{levels-required} shows that $c$ has been chosen large enough that $c$-contributions matter much more than 1-contributions.

The intended embedding of $\mathscr M_0$ is illustrated below,
with the symbols $\mt{0}$ and $\mt{2}$ in red and $\mt{1}$ in blue, and defined formally in \Cref{df:intended}.

\includegraphics[width=10cm]{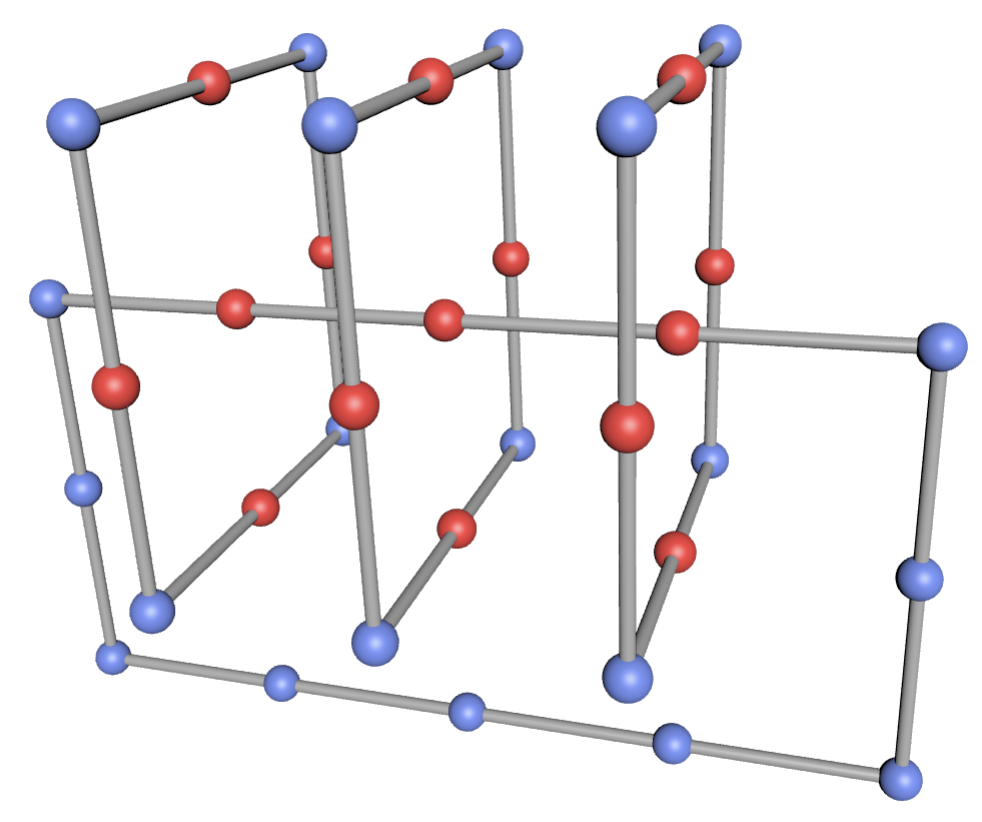}

\begin{definition}[Intended embedding of $\mathscr M_0$]\label{df:intended}
The intended embedding of the four-word complex $\mathscr M_0$ consists of three embeddings of $(01)^4$ and one embedding of $2^31^9$.
These are functions $f_i: \mathbb Z/8\mathbb Z \to \mathbb Z^3$, $i=-1,0,1$,
and $f_2 : \mathbb Z/12\mathbb Z \to \mathbb Z^3$.
In the intended embedding, the range of $f_2$ is the rectangle with corners $(-2,0,0), (2,0,0), (-2,0,-2), (2,0,-2)$ and
$(f_2(0), f_2(1), f_2(2))=((-1,0,0), (0,0,0), (1,0,0))$.
The range of $f_k$ for $k\in\{-1,0,1\}$ is the square with corners $(k,\pm 1,\pm 1)$, with the range of $f$ on $\{1,3,5,7\}$ being the corners.
\end{definition}

\begin{definition}\label{df:bent}
A set $\{x,y,z\}\subseteq\mathbb Z^3$ with $d_1(x,y)=d_1(y,z)=1$ and $x\ne z$ (and hence  $d_1(x,z)=2$) is
\emph{straight} if $y=\frac{x+z}2$ and \emph{bent} otherwise.
\end{definition}

\begin{definition}\label{df:Phi}
	Let $x,y\in\mathbb Z^3$. The $L^1$ distance $d_1$ between $x$ and $y$ is
	\[
		d_1(x,y)=\sum_{i=0}^2 |x(i)-y(i)|.
	\]
	Let $A\subseteq\mathbb Z^3$. The $L^1$ distance between $x$ and $A$ is
	\[
		d(x,A) = \min\{d_1(x,y)\mid y \in A\}.
	\]
	The vertex-boundary of $A$, $\Phi(A)$ (in the terminology of Harper, Section 4.1.1) is
	\[
		\Phi(A) = \{x\mid d(x,A)=1\}.
	\]
\end{definition}

\begin{lemma}\label{fourteen-neighbors}
	Let $A = \{(-1,0,0), (0,0,0), (1,0,0)\} \subseteq \mathbb Z^3$.
	The cardinality of $\Phi(A)$ (as in \Cref{df:Phi}) is 14.
\end{lemma}
\begin{proof}
	For each $x\in\Phi(A)$ there is a $y\in A$ such that $x(i)\ne y(i)$ for some $i$
	and $x(j)=y(j)$ whenever $j\ne i$.
	If $i=0$ then $y\in\{(-2,0,0), (2,0,0)\}$.
	If $i=1$ then $y=(a,b,0)$, and
	if $i=2$ then $y=(a,0,b)$, for some $a\in\{-1,0,1\}$ and $b\in\{-1,1\}$.
	Thus, there are $3\times 2\times 2+2=14$ such points $x$ overall.
\end{proof}

\begin{lemma}\label{thirteen-neighbors}
	Let $A = \{(0,1,0), (0,0,0), (1,0,0)\} \subseteq \mathbb Z^3$.
	The cardinality of $\Phi(A)$ is 13.
\end{lemma}
\begin{proof}
	For each $x\in\Phi(A)$ there is a $y\in A$ such that $x(i)\ne y(i)$ for some $i$
	and $x(j)=y(j)$ whenever $j\ne i$.
	If $i=2$ then $y\in\{(a,0,b)\mid a\in\{-1,0,1\},b\in\{\pm 1\}\}$.
	Otherwise, $y$ is one of the seven points indicated by $*$ below, where elements of $A$ are indicated by $a$.

\begin{tikzpicture}[nodes={font=\large}]
\matrix[matrix of nodes,
        nodes in empty cells,
        row sep=-\pgflinewidth,
        column sep=-\pgflinewidth,
        nodes={minimum size=7mm,anchor=center,draw=black}]{
   & * & &  \\
 * & a & * & \\
 * & a & a & * \\
   & * & * &\\
};
\end{tikzpicture}

Thus, there are $6+7=13$ such points $x$ overall.
\end{proof}

\begin{lemma}\label{urgentcare-sep8}
	Let $A = \{(-1,0,0), (0,0,0), (1,0,0)\} \subseteq \mathbb Z^3$.
	Suppose $f: (\mathbb Z / 8\mathbb Z) \to \mathbb Z^3 \setminus A$ is injective, and
	$\|f([x])-f([x+1])\|_1=1$ for all $x\in\mathbb Z$.
	Suppose that the $L^1$ distance from $f(2x)$ to $A$ is $d(f(2x),A)=1$ for all $x$.
	Moreover, suppose that $\|f(x_0) - (0,0,0)\|_1=1$ for some $x_0$.
	Then the image of $f$ is
	\[
		\{(0,x,y)\in\mathbb Z^3 \mid |x|\le 1, |y|\le 1, (x,y)\ne (0,0)\}.
	\]
\end{lemma}
\begin{proof}
	By symmetry, we may assume that $f(x_0) = (0,1,0)$ and $x_0=0$. Since $d(f(2),A)=1$ and $d(f(0),f(1))=d(f(1),f(2))=1$,
	$f(2)$ must be either $(0,0,1)$ or $(0,0,-1)$. By symmetry assume $f(2) = (0,0,1)$.
	Then by process of elimination $f(4)$ must be $(0,-1,0)$.
	Then $f(6)$ must be $(0,0,-1)$ and by interpolation we are done.
\end{proof}

\begin{lemma}\label{levels-required}
	Suppose we are given an embedding $\psi$ of the multiset of words $\mathscr M_m$.
	Let $\phi$ be the intended embedding.
	Suppose that $c\ge 10$ and that $\phi$ has strictly more $c$-contributions than $\psi$.
	Then the score of $\phi$ under $h_c$ is strictly greater than the score of $\psi$ under $h_c$.
\end{lemma}
\begin{proof}
	The score of $\phi$ is the sum of the $c$-contributions and the 1-contributions.
	Each chain of length 8 in $\mathscr M_m$ has all its $\mt 0$s in a fixed congruence class mod 2
	under the map
	\[
		x\in \mathbb Z^3 \mapsto x(0)+x(1)+x(2)\text{ (mod 2).}
	\]
	Therefore, two of these three chains (the ``unlucky two'') do not achieve any 1-contributions with each other.

	Each $\mt 0$ has two neighbors that are adjacent in its chain and $N=6-2$ other neighboring lattice points.
	The part of the score arising from 1-contributions involving the lucky chain is then at most $z \times s \times N$
	where $z=4$ is the number of $\mt 0$s in the lucky chain, and $s=1+1=2$ is the number of points arising from a single 1-contribution.
	As a first approximation, the maximum score from 1-contributions is therefore $4\times 2\times 4 = 32$.

	However, some of these potential 1-contributions may instead be $c$\--contributions.
	For each $i$, let $s_i$ be the portion of the score coming from 1-contributions involving a lucky $\mt 0$ involved in $i$ many $c$-contributions.
	Let $t$ be the portion of the score coming from $c$-contributions.
	Then the maximum score obtainable is at most $t+\sum_{i=0}^2 s_i$.

	Next, let $x$ be the number of missing $c$-contributions in $\psi$, i.e., the number of $c$-contributions of $\phi$
	minus the number of $c$-contributions of $\psi$.
	By assumption $x\ge 1$.
	If $x\le 4$ then the maximum score obtainable is at most
	\begin{equation}\label{bound}
		t + s_1 + s_0 = (12-x)(c+1) + (4-x)\times 2\times 3 + x \times 2 \times 4
	\end{equation}
	assuming that each $\mt 0$ is only involved in at most one $c$-contribution.

	If the $\mt 2$s are bent then
	in principle one special $\mt 0$ could be involved in two $c$-contributions.
	In that case only $13-2=11$ $c$-contributions are possible by \Cref{thirteen-neighbors} and the bound becomes
	\[
			(11-x)(c+1) + s_2 + s_1 + s_0 = (11-x)(c+1) + 1\times 2\times 2  + (3-x)\times 2\times 3 + x \times 2 \times 4
	\]
	which is smaller than \eqref{bound} and hence can be ignored.

	The amount \eqref{bound} is strictly less than the score for the intended folding,
	\(
		12(c+1) + 4\times 2\times 2
	\)
	whenever
	\[
		c+1 > 8/x +2.
	\]
	Here the right-hand side is maximized at $x=1$, so that we need $c \ge 10$.

	If $x>4$ then the maximum score obtainable is at most
	\[
		(12-x)(c+1) + 4 \times 2 \times 4
	\]
	so in this case we only require $c\ge 4$.
\end{proof}

\begin{remark}
The reader may wonder whether a factor of ten for hydrophobicity levels is excessive.
However, White~\cite{white96}\footnote{See also \url{https://blanco.biomol.uci.edu/hydrophobicity_scales.html}} found that
in a certain \emph{interface scale}, the amino acid Glu$^-$ (glutamate)
has a $\Delta G_{wif}$ (kcal/mol) experimentally determined hydrophobicity value of 2.02, and Gly (glycine) has the value 0.01,
giving a factor of 202.
\end{remark}

\subsection{Topological conclusions}
To properly formulate our conclusions we recall some notions from topology.
\begin{definition}\label{isotopy}
Let $N$ and $M$ be manifolds and $g$ and $h$ be embeddings of $N$ in $M$.  A continuous map 
$F:M \times [0,1] \rightarrow M $ 
is defined to be an \emph{ambient isotopy} taking $g$ to $h$ if
each $F_t: M \rightarrow M, F_t(\cdot) = F(\cdot, t)$ is a homeomorphism from $M$ to itself,
$F_0$ is the identity function and $F_1 \circ g = h$.
\end{definition}
We are interested in \Cref{isotopy} in the case where $N$ is a disjoint union of two circles and $M$ is $\mathbb R^3$.
Note that the embedding $h$ may well take a circle to a rectangle or any other homeomorph of a circle.

\begin{definition}
	Let $a_0,a_1\in \mathbb N$.
	Let $f_i:\mathbb Z/a_i\mathbb Z\to\mathbb Z^3$ be injective functions with the property that
	$d(f_{i}([x]),f_{i}([x+1]))=1$ for all $x$ and both $i$.
	The union $C_i$ of all the line segments from $f_{i}([x])$ to $f_{i}([x+1])$ for $x\in\mathbb Z$ is then homeomorphic to a circle.
	If $C_0$ and $C_1$ form a proper link in $\mathbb R^3$
	(that is, there is no ambient isotopy mapping $C_0$ and $C_1$ to circles (or and kind of knots) that are separated by a plane)
	then we say, by extension, that $f_0$ and $f_1$ are \emph{linked}.
\end{definition}

\begin{theorem}\label{link-main}
	The maximum score of all closed-chain embeddings of $\mathscr M_m$, for any $m\ge 0$,
	is only achieved for embeddings for which for some $j$, the range of $f_{(\mt{01})^4,j}$ and the range of $f_{\mt{2}^3\mt{1}^{9+2m},1}$
	are linked.
\end{theorem}
\begin{proof}
Consider any optimal embedding.
In it, the embedded chain $\mt{2}^3$ can either be bent (\Cref{df:bent}) or not, as illustrated on the left and right, respectively, below.

\begin{tikzpicture}[nodes={font=\large}]
\matrix[matrix of nodes,
        nodes in empty cells,
        row sep=-\pgflinewidth,
        column sep=-\pgflinewidth,
        nodes={minimum size=7mm,anchor=center,draw=black}]{
  $\mt{2}$  \\
  $\mt{2}$ & $\mt{2}$   \\
};
\end{tikzpicture}
\begin{tikzpicture}[nodes={font=\large}]
\matrix[matrix of nodes,
        nodes in empty cells,
        row sep=-\pgflinewidth,
        column sep=-\pgflinewidth,
        nodes={minimum size=7mm,anchor=center,draw=black}]{
  $\mt{2}$ & $\mt{2}$ & $\mt{2}$ \\
};
\end{tikzpicture}

Let $f_2 = f_{\mt{2}^3\mt{1}^{9+2m},1}$.
Let $A=\{f_2(0),f_2(1,),f_2(2)\}$.
If $A$ is straight then $A$ has 14 neighbors,
of which 12 can contribute to the score and the last 2 are occupied by $f(3)$ and $f(11)$.
See \Cref{fourteen-neighbors}.

If $A$ is bent then $A$ has only 13 neighbors (\Cref{thirteen-neighbors}).
Therefore, any such $f$ for which $A(f)$ is bent yields a lower score than the intended, linked, embedding from \Cref{df:intended}
by \Cref{levels-required}.

Next, for an unbent chain we claim that any optimal fold has the middle $\mt{2}$ in $\mt{222}$ wrapped in a link.

Indeed, if a $\mt{01010101}$ has all four $\mt{0}$s next to $\mt{2}$s in an unbent $\mt{222}$ chain,
and one of these $\mt 0$s is next to the middle $\mt 2$ in the chain, then the $\mt{01010101}$
must be wrapped circularly around the middle $\mt{2}$ by \Cref{urgentcare-sep8}.
\end{proof}

\begin{remark}
	\Cref{link-main} does not claim that all the three 8-chains must be linked with the 12-chain in $\mathscr M_0$,
	and in fact we can show that there is an optimal embedding where only the middle 8-chains is linked with the 12-chain.
\end{remark}

\subsection*{Future work}
\Cref{link-main} opens the door to exploration of the extent of optimal links in HP models.
Finding these by brute force computation may be challenging, but our results indicate how combinatorial reasoning can be used instead.
Natural questions include:
\begin{enumerate}
	\item Is there a multiset $\mathscr M$ of cardinality 2 and a levels-of-hydrophobicity function $h$ such that
		all $h$-optimal embeddings of $\mathscr M$ make the two words linked?
	\item Is there a single word $w$ and function $h$ such that all $h$-optimal embeddings of $\{w\}$ are knotted?
\end{enumerate}

\subsection*{Acknowledgments}
This work was partially supported by grants from the Simons Foundation (\#704836
to Bj{\o}rn Kjos-Hanssen) and Majesco Inc.~(University of Hawai\textquoteleft i Foundation
Account \#129-4770-4).

\bibliographystyle{plain}
\bibliography{protein}
\end{document}